\documentclass{amsart}
\usepackage[latin1]{inputenc}
\usepackage{color}
\usepackage{amssymb}
\usepackage[all]{xy}
\newtheorem{thm}{Theorem}[section]
\newtheorem{cor}[thm]{Corollary}
\newtheorem{lem}[thm]{Lemma}
\newtheorem{prop}[thm]{Proposition}
\theoremstyle{definition}
\newtheorem{defn}[thm]{Definition}
\newtheorem{rem}[thm]{Remark}
\newtheorem{prob}[thm]{Problem}
\newtheorem{example}[thm]{Example}
\numberwithin{equation}{section}
\newcommand{\nat}{\mathbb{N}}

\newcommand{\Nat}{\mathbb{N}}

\def\epsilon{\varepsilon}
\newcommand{\sub}{\subseteq}
\newcommand{\erre}{\mathbb R}

\title{The free Banach lattice generated by a Banach space}
\author[A. Avil\'es]{Antonio Avil\'es}
\address{Departamento de Matem\'{a}ticas, Facultad de Matem\'{a}ticas, Universidad de Murcia, 30100 Espinardo (Murcia), Spain} 
\email{avileslo@um.es}

\author[J. Rodr\'{i}guez]{Jos\'e Rodr\'{i}guez}
\address{Departamento de Ingenier\'{i}a y Tecnolog\'{i}a de Computadores,
Facultad de Inform\'{a}tica, Universidad de Murcia, 30100 Espinardo (Murcia), Spain}  
\email{joserr@um.es}

\author[P. Tradacete]{Pedro Tradacete}
\address{Departamento de Matem\'{a}ticas, Universidad Carlos III de Madrid, 28911 Legan\'es (Madrid), Spain}
\email{ptradace@math.uc3m.es}

\subjclass[2010]{46B42, 46B50}

\keywords{Banach lattice; free lattice; Nakano property; order interval; weakly compactly generated space}

\begin{document}

\setcounter{tocdepth}{1}
\date{\today}

\begin{abstract}
The free Banach lattice over a Banach space is introduced and analyzed. This generalizes the concept of free Banach lattice over a set of generators, and allows us to study the Nakano property and the density character of non-degenerate intervals on these spaces, answering some recent questions of B. de Pagter and A.W. Wickstead. 
Moreover, an example of a Banach lattice which is weakly compactly generated as a lattice but not as a Banach space is exhibited, thus answering a question of J. Diestel.
\end{abstract}

\maketitle

\section{Introduction}

The purpose of this paper is to introduce the free Banach lattice generated by a Banach space and investigate its properties. The free Banach lattice generated by a 
set $A$ with no extra structure, which is denoted by $FBL(A)$, has been recently introduced and analyzed by B.~de Pagter and A.W.~Wickstead in \cite{dePW}.
Namely, $FBL(A)$ is a Banach lattice together with a bounded map $u:A \to FBL(A)$ having the following universal property: for every
Banach lattice $Y$ and every bounded map $v:A \to Y$ there is a unique lattice homomorphism 
$S: FBL(A)\to Y$ such that $S\circ u=v$ and $\|S\|=\sup\{\|v(a)\|:a\in A\}$. Our aim here is to 
provide an analogous construction replacing the set $A$ by a Banach space $E$, in a way that the resulting 
free Banach lattice behaves well with respect to the Banach space structure of $E$. 

In the absense of topology, the free vector lattice generated by a set $A$ was previously considered in \cite{Baker,Bleier}
and can be characterized as certain sublattice of~$\mathbb R^{\mathbb R^A}$. Constructing the free Banach lattice $FBL(A)$ 
becomes tantamount to finding the largest possible lattice norm that the free vector lattice over $A$ can carry. 
Among other things, the existence of such a norm is proved in \cite{dePW}. However, one can provide an explicit form of the norm of $FBL(A)$ (see Corollary \ref{c:FBL(A)}). 

Loosely speaking, the free Banach lattice $FBL[E]$ generated by a Banach space~$E$ is a Banach lattice which 
contains a subspace linearly isometric to~$E$ in a way that its elements work as lattice-free generators. 
In other words, the subspace of generators has two properties: first, the generators carry no lattice relation among them 
(except for the linear and metric ones coming from~$E$), and second, the sublattice spanned by these generators is dense in the free Banach lattice. 
To be more precise, if $\phi_E$ stands for the canonical isometric embedding of~$E$ into~$FBL[E]$, then the universal property of~$FBL[E]$ reads as follows:
for every Banach lattice $X$ and every operator $T:E\rightarrow X$ there exists a unique lattice homomorphism 
$\hat T:FBL[E]\rightarrow X$ such that $\|\hat T\|=\|T\|$ and $\hat T\circ \phi_E=T$, i.e. the following diagram commutes:
$$
	\xymatrix{E\ar_{\phi_E}[d]\ar[rr]^T&&X\\
	FBL[E]\ar_{\hat{T}}[urr]&& }
$$ 

Section~\ref{section:description} is devoted to constructing the free Banach lattice~$FBL[E]$
generated by a Banach space $E$. The construction will be achieved by defining a norm on a sublattice 
of~$\mathbb R^{E^*}$. This explicit description of the norm in the free Banach lattice is a helpful tool to tackle some questions 
that were raised in~\cite{dePW}. It should come as no great surprise that as~$\ell_1(A)$ is the free Banach space over the set~$A$, 
then $FBL(A)=FBL[\ell_1(A)]$ (see Corollary \ref{c:FBL(A)}). In particular, the free 
Banach lattice generated by a Banach space can be thought of as a generalization of the free Banach lattices of the form~$FBL(A)$.  

In Sections~\ref{s:OrderIntervals} and~\ref{s:nakano} we discuss further properties of the free Banach lattice generated by a Banach space.
In \cite[Theorem 8.3]{dePW} the authors show that, given any infinite set~$A$, the smallest cardinal $\mathfrak a$ such that 
every order interval in $FBL(A)$ has density character at most $\mathfrak a$ is $|A|$ (the cardinality of~$A$). 
They ask whether all non-degenerate order intervals in $FBL(A)$ must have the same density character
(that would necessarily be equal to~$|A|$). We will see that this is indeed the case and, more generally, 
that for any Banach space~$E$, every non-degenerate order interval in $FBL[E]$ has the same density character as~$E$ (Theorem~\ref{t:orderintervals}).
Another intriguing question raised in~\cite{dePW} is whether the norm of a free Banach lattice
of the form~$FBL(A)$ must be Fatou, or even Nakano. We will show that this is indeed the case (Theorem~\ref{thm:Nakano}), while 
this property is not shared by all free Banach lattices generated by a Banach space (Theorem~\ref{t:nofatou}).

Finally, in Section~\ref{s:wcg} we revisit a question that J. Diestel raised 
during the conference ``Integration, Vector Measures and Related Topics IV'' held in La Manga del Mar Menor, Spain, 2011.
A Banach lattice $X$ is said to be {\em lattice weakly compactly generated} (LWCG for short) if there is a weakly compact set $L\sub X$ such that the
sublattice generated by~$L$ is dense in~$X$. This is formally weaker than being weakly compactly generated (WCG for short) 
as a Banach space, which means that there is a weakly compact set $K \sub X$ such that $X=\overline{{\rm span}}(K)$.

\begin{prob}[Diestel]\label{problem:Diestel}
Is every LWCG Banach lattice WCG?
\end{prob}

This and related questions have been recently investigated in~\cite{AGLRT}, where Problem~\ref{problem:Diestel} is solved 
affirmatively for Banach lattices which are order continuous or have weakly sequentially continuous lattice operations.
Here we will provide a negative answer to Diestel's question by showing 
that the free Banach lattice $FBL[\ell_2(\Gamma)]$ is LWCG but not WCG as long as $\Gamma$ is uncountable (Corollary~\ref{c:Diestel}). 

\subsection*{Terminology}
We only consider linear spaces over the real field. Given a Banach space~$E$, its norm is denoted by $\|\cdot\|_E$ or simply
$\|\cdot\|$ if no confussion arises. The closed unit ball and the unit sphere of~$E$ are denoted by~$B_E$ and~$S_E$, respectively. 
The linear subspace generated by a set $S \sub E$ is denoted by~${\rm span}(S)$ and its closure is denoted by $\overline{{\rm span}}(S)$.
The symbol $E^*$ stands for the (topological) dual of~$E$. Given a Banach lattice~$X$, we write $X_+=\{x\in X: x\geq 0\}$.
By an {\em operator} between Banach spaces we mean a linear continuous map.

\section{A description of the free Banach lattice}\label{section:description}

Throughout this section $E$ is a Banach space.
Our first aim is to show that the free Banach lattice generated by $E$ exists and provide an explicit description
(Theorem~\ref{t:fblb}).

We denote by $H[E]$ the linear subspace of $\erre^{E^*}$
consisting of all positively homogeneous functions $f:E^\ast \to \mathbb{R}$.
For any $f\in H[E]$ we define
$$
	\|f\|_{FBL[E]} := 
	\sup\left\{\sum_{k=1}^n |f(x_k^\ast)| : \, n\in\mathbb N, \, x_1^*,\dots,x_n^*\in E^*, \,  \sup_{x\in B_E} \sum_{k=1}^n |x_k^\ast(x)|\leq 1\right\}.
$$

\begin{rem}\label{rem:PointwiseBounded}
If $f\in H[E]$, then 
$$
	\|f\|_\infty:=\sup\big\{|f(x^*)|: \, x^*\in B_{E^*}\big\} \leq \|f\|_{FBL[E]}.
$$
\end{rem}

It is routine to check that $H_0[E]:=\{f\in H[E]: \|f\|_{FBL[E]}<\infty\}$
is a Banach lattice when equipped with the norm~$\|\cdot\|_{FBL[E]}$
and the pointwise lattice operations. 

\begin{defn}\label{definition:FBL}
Given any $x\in E$, let
$\delta_x\in H_0[E]$ be defined by 
$$
	\delta_x(x^\ast):= x^\ast(x) 
	\quad\mbox{for all }x^*\in E^*.
$$
We define $FBL[E]$ to be the closed sublattice of $H_0[E]$ generated by $\{\delta_x:x\in E\}$.
\end{defn}

The following lemma is straightforward.

\begin{lem}\label{l:isometry}
The mapping $\phi_E: E\rightarrow FBL[E]$ given by $\phi_E(x):=\delta_x$ defines a linear isometry between $E$ and its image in $FBL[E]$.
\end{lem}

\begin{thm}\label{t:fblb}
Let $X$ be a Banach lattice and $T:E\rightarrow X$ an operator. There is a unique lattice homomorphism $\hat T:FBL[E]\rightarrow X$ 
extending $T$ (in the sense that $\hat T \circ \phi_E=T$). Moreover, $\|\hat T\|=\|T\|$.
\end{thm}
\begin{proof}
Recall that the {\em free vector lattice} over the set~$E$, denoted by~$FVL(E)$, is the 
vector sublattice of~$\erre^{\erre^E}$ generated by the family $\{\eta_{x}: x\in E\}$, where
$$
	\eta_x(f):=f(x) \quad\mbox{for all }f\in \erre^E
$$
(\cite{Bleier}, cf. \cite[Theorem 3.6]{dePW}). Thus, the 
mapping $\phi_E:E \to FBL[E]$ of Lemma~\ref{l:isometry} induces a lattice homomorphism $\varphi: FVL(E) \to FBL[E]$
such that $\varphi(\eta_x)=\delta_x$ for all $x \in E$. In particular, $\varphi$ has dense range. The 
universal property of~$FVL(E)$ can be used again to obtain
a lattice homomorphism $\tilde T:FVL(E)\rightarrow X$ such that $\tilde T(\eta_{x})=T(x)$ for all $x\in E$.

{\em Claim.} For every $f\in FVL(E)$ we have
\begin{equation}\label{eqn:Domination}
	\|\tilde T(f)\|_X \leq \| T\| \|\varphi(f)\|_{FBL[E]}.
\end{equation}  
Once the claim is proved the proof of the theorem finishes as follows. Inequality~\eqref{eqn:Domination} and
the density of $\varphi(FVL(E))$ in~$FBL[E]$ allow us to define an operator $\hat T: FBL[E] \to X$ with $\|\hat T\|\leq \|T\|$ such that $\hat T \circ \varphi =\tilde T$.
Since $\varphi$ and~$\tilde T$ are lattice homomorphisms, so is $\hat T$. Clearly, $\hat T \circ \phi_E=T$.
Moreover, we have $\|\hat T\|=\|T\|$, because $\|T(x)\| =\|\hat T(\delta_x)\| \leq \|\hat T\| \|\delta_x\|_{FBL[E]}=\|\hat T\|\|x\|$ for every $x\in E$.
For the uniqueness of~$\hat T$, bear in mind that any lattice homomorphism from~$FBL[E]$ to a Banach lattice is uniquely determined
by its values in $\{\delta_x:x\in E\}$.

{\em Proof of Claim.} The case $T=0$ is trivial, so we assume that $T\neq 0$. Fix $f\in FVL(E)$. Actually, $f$ belongs to the sublattice of~$FVL(E)$ generated by 
$\{\eta_{x_1},\dots, \eta_{x_n}\}$ for some finite set $\{x_1,\ldots, x_n\}\sub E$. 
By \cite[Ex. 8, p. 204]{AB}, we can write 
$$
	f=\bigvee_{i=1}^m f_i-\bigvee_{j=1}^p g_j
$$
for some $f_1,\dots,f_m, g_1,\dots,g_p$ in ${\rm span}(\{\eta_{x_1},\ldots,\eta_{x_n}\})$.
Write $f_i=\sum_{l=1}^{n}\lambda_l^i\eta_{x_l}$ and $g_j= \sum_{l=1}^{n}\mu_l^j\eta_{x_l}$ 
for some $\lambda_l^i,\mu_l^j \in \erre$. Since $\tilde T$ is a lattice homomorphism, we have 
\begin{equation}\label{eqn:Ttilde}
	\tilde T(f)=\bigvee_{i=1}^m u_i-\bigvee_{j=1}^p v_j
\end{equation}
where $u_i:=\sum_{l=1}^{n}\lambda_l^iT(x_l)$ and $v_j:=\sum_{l=1}^{n}\mu_l^jT(x_l)$.

Since $\tilde T$ and $\varphi$ are lattice homomorphisms, it suffices to check~\eqref{eqn:Domination}
in the particular case that $f\geq 0$. In this case, \eqref{eqn:Domination} is equivalent to
the fact that $y^*(\tilde T(f)) \leq \|T\|\|\varphi(f)\|_{FBL[E]}$ 
for every $y^* \in B_{X^*}\cap (X^*)_+$ (bear in mind that $\tilde{T}(f)\geq 0$).
Fix $y^* \in B_{X^*}\cap (X^*)_+$.
Take an arbitrary decomposition $y^*=\sum_{k=1}^m y_k^*$ where
$y_1^*,\dots,y_m^*\in (X^*)_+$. Let us define $x_k^*:=\|T\|^{-1} T^*(y_k^*)\in E^*$ for every $k\in \{1,\dots,m\}$, which satisfy 
$$
	\sup_{x\in B_E} \sum_{k=1}^m |x_k^*(x)|
	\leq
	\sup_{x\in B_E} \sum_{k=1}^m y_k^*\Big(\frac{1}{\|T\|}|T(x)|\Big)
	=
	\sup_{x\in B_E} y^*\Big(\frac{1}{\|T\|}|T(x)|\Big)\leq\|y^*\|\leq1.
$$
Hence, 
\begin{eqnarray*}
	\|\varphi(f)\|_{FBL[E]}&\geq&\sum_{k=1}^m \varphi(f)(x^*_k)=\sum_{k=1}^m \Bigg(\bigvee_{i=1}^m \varphi(f_i)(x_k^*)-\bigvee_{j=1}^p \varphi(g_j)(x_k^*)\Bigg)\\
	&=&\frac1{\|T\|}\sum_{k=1}^m \Bigg(\bigvee_{i=1}^m y^*_k(u_i)-\bigvee_{j=1}^p y^*_k(v_j)\Bigg)\\
	&\geq&\frac1{\|T\|}\sum_{k=1}^m \Bigg(y^*_k(u_k)- y^*_k\Big(\bigvee_{j=1}^pv_j\Big)\Bigg)\\
	&=&\frac1{\|T\|}\Bigg(\sum_{k=1}^m y^*_k(u_k)- y^*\Big(\bigvee_{j=1}^pv_j\Big)\Bigg). 
\end{eqnarray*}
By the classical Riesz-Kantorovich formulas (cf. \cite[Theorem 1.18]{AB}) we have
$$
	y^*\Big(\bigvee_{k=1}^m u_k\Big)=\sup\Big\{\sum_{k=1}^m y^*_k(u_k): \, y_k^*\in (X^*)_+, \,y^*=\sum_{k=1}^m y^*_k\Big\}.
$$
Hence, by taking supremum over all such decompositions of~$y^*$, the above inequalities yield
\begin{equation}\label{eqn:final}
	y^*(\tilde T(f))\stackrel{\eqref{eqn:Ttilde}}{=}y^*\Big(\bigvee_{i=1}^m u_i-\bigvee_{j=1}^p v_j\Big)\leq \|T\|\|\varphi(f)\|_{FBL[E]},
\end{equation}
as desired. The proof is complete. 
\end{proof}

\begin{cor}\label{cor:FBLofBanachLattice}
If $E$ is a Banach lattice, then it is the range of a lattice projection $P:FBL[E]\rightarrow E$.
\end{cor}
\begin{proof}
Apply Theorem \ref{t:fblb} to the identity $T:E\rightarrow E$ to get $P:=\hat T$.
\end{proof}

\begin{cor}\label{cor:RealHomomorphisms}
A functional $\varphi\in FBL[E]^*$ is a lattice homomorphism if and only if there is $x^*\in E^*$ such that
$\varphi(f)= f(x^*)$ for all $f\in FBL[E]$.
\end{cor}
\begin{proof}
Given any $x^*\in E^*$, the evaluation functional 
$$
	\varphi_{x^*}:FBL[E]\to \erre, \quad
	\varphi_{x^*}(f):=f(x^*),
$$
is obviously a lattice homomorphism. Conversely, suppose $\varphi\in FBL[E]^*$ is a lattice homomorphism. 
Then $x^*:=\varphi\circ \phi_E$ belongs to $E^*$. Since
$\varphi$ and $\varphi_{x^*}$ are lattice homomorphisms
such that $\varphi\circ \phi_E=\varphi_{x^*}\circ \phi_E$, the uniqueness part of
Theorem~\ref{t:fblb} yields that $\varphi=\varphi_{x^*}$.
\end{proof}

In the spirit of \cite[Corollary 4.10]{dePW}, we have the following:

\begin{cor}\label{cor:subspace}
Let $F\sub E$ be a closed subspace which is complemented by a contractive projection. Then:
\begin{enumerate}
\item[(i)] $FBL[F]$ is isometrically order isomorphic to a closed sublattice of~$FBL[E]$.
\item[(ii)] $FBL[F]^*$ is isometrically order isomorphic to a $w^*$-closed band of $FBL[E]^*$.
\end{enumerate}
\end{cor}
\begin{proof}
Let $P:E\rightarrow F$ be a contractive projection. Consider 
$$
	T:=\phi_F\circ P:E\rightarrow FBL[F]
$$ 
and let $\hat T:FBL[E]\rightarrow FBL[F]$ be the unique lattice homomorphism extending~$T$, which
satisfies $\|\hat T\|=\|T\|=\|P\|=1$ (Theorem~\ref{t:fblb}). Let $i:F \to E$ be the canonical inclusion, consider
$$
	S:=\phi_E\circ i:F\rightarrow FBL[E]
$$ 
and let $\hat S: FBL[F] \to FBL[E]$ be the unique lattice homomorphism extending~$S$,
which also satisfies $\|\hat S\|=\|S\|=\|i\|=1$ (Theorem~\ref{t:fblb}). For every $x\in F$ we have 
$$
	\hat S(\delta_x)=\phi_E(i(x))=\delta_x
	\quad\mbox{and}\quad \hat{T}(\delta_x)=\phi_F(P(x))=\phi_F(x)=\delta_x,
$$
so $\hat T \circ \hat S$ is the identity on~$FBL[F]$. It follows that $\hat S$ is an isometric embedding
(which yields statement~(i)) and that $\hat{T}$ is a contrative projection onto $FBL[F]$. Statement~(ii)
now follows from \cite[Proposition 4.9]{dePW}.
\end{proof}

In the particular case $E=\ell_1(A)$ the space $FBL[E]$ turns out to be the free
Banach lattice generated by the set~$A$ (in the sense of~\cite{dePW}), as we next show.  

\begin{cor}\label{c:FBL(A)}
Let $A$ be a non-empty set. Then:   
\begin{enumerate}
\item[(i)] For every $f\in H[\ell_1(A)]$ we have
$$
	\|f\|_{FBL(A)}=\sup\left\{\sum_{k=1}^n |f(x^*_k)| : \, n\in\mathbb N, \, x_1^*,\dots,x_n^* \in\ell_\infty(A), \, \sup_{a\in A} \sum_{k=1}^n |x_k^*(a)|\leq 1\right\}.
$$
\item[(ii)] $FBL[\ell_1(A)]$ is the closed sublattice of $H_0[\ell_1(A)]$ generated by the family $\{\delta_{e_a}:a\in A\}$, where
$(e_a)_{a\in A}$ is the unit vector basis of~$\ell_1(A)$.
\item[(iii)] The pair $(FBL[\ell_1(A)],i)$ is the free Banach lattice generated by~$A$, where $i:A\rightarrow FBL[\ell_1(A)]$ 
is the bounded map given by $i(a):=\delta_{e_a}$ for all $a\in A$.
\end{enumerate}
\end{cor}
\begin{proof}
(i) is elementary, while (ii) is an immediate consequence of~Lemma~\ref{l:isometry}
and the fact that $\ell_1(A)=\overline{{\rm span}}(\{e_a:a\in A\})$.

In order to check (iii), fix a Banach lattice $Y$ and a bounded map $\kappa:A\rightarrow Y$.
Consider the operator $T:\ell_1(A)\rightarrow Y$ satisfying $T(e_a)=\kappa(a)$ for all $a\in A$.
By Theorem~\ref{t:fblb}, there is a lattice homomorphism $\hat T: FBL[\ell_1(A)]\rightarrow Y$ such that $\hat T \circ \phi_{\ell_1(A)}=T$
(hence $\hat T \circ i=\kappa$), with 
$$
	\|\hat T\|=\|T\|=\sup\{\|\kappa(a)\|: \, a\in A\}.
$$
Moreover, by~(ii), any lattice homomorphism $S: FBL[\ell_1(A)]\rightarrow Y$ such that
$S \circ i=\kappa$ must coincide with~$\hat T$. 
\end{proof}

The following examples show that $H[\ell_1] \supsetneq H_0[\ell_1] \supsetneq FBL[\ell_1]$.

\begin{example}
Let $f\in H[\ell_1]$ be defined by $f(x) := \sup\big\{\frac{|x(a)|}{a} : a\in \mathbb{N}\big\}$ for all $x\in \ell_\infty$.
Then $\|f\|_{FBL[\ell_1]} = \infty$. Indeed, let $(e^*_n)_{n\in \nat}$ be the unit vector basis of~$c_0$. For each $n\in \Nat$
we have $\sup_{a\in \nat}\sum_{k=1}^n|e^*_k(a)| \leq 1$, so Corollary~\ref{c:FBL(A)}(i) yields
$$
	\|f\|_{FBL[\ell_1]}\geq \sum_{k=1}^n f(e^*_k)= \sum_{k=1}^n \frac{1}{k}.
$$
By taking limits when $n\to \infty$ we get $\|f\|_{FBL[\ell_1]} = \infty$. \qed
\end{example}

\begin{example}(A function in $H_0[\ell_1]\setminus FBL[\ell_1]$.)\label{s:examples}
Define a positively homogeneous
function $f:\ell_\infty \to \erre$ by
$$
	f(x) := \min\left(|x(1)|, \, \sup\left\{\frac{|x(a)|}{a} :\, a\geq 2\right\}\right)
	\quad\mbox{for all }x\in \ell_\infty.
$$
The fact that $0 \leq f(x) \leq |x(1)|$ for all $x\in \ell_\infty$ implies that $\|f\|_{FBL[\ell_1]} \leq 1$. 
We will prove that $f\not\in FBL[\ell_1]$ by showing that $\|f-g\|_{FBL[\ell_1]} \geq \frac{1}{4}$
for every $g\in H_0[\ell_1]$ which belongs to the sublattice generated by~$\{\delta_{e_a}:a\in \nat\}$
(bear in mind Corollary~\ref{c:FBL(A)}(ii)).
To this end, note first that such $g$ belongs to the sublattice generated by $\{\delta_{e_1},\dots,\delta_{e_n}\}$
for some~$n\in \nat$. For each $k\in \{1,\dots,n\}$, we define $x^*_k,y^*_k\in \ell_\infty$ by declaring
$x^*_k(1) = y^*_k(1) := \frac{1}{n}$, $x^*_k(n+k) := 1$, and all other coordinates of $x^*_k$ and $y^*_k$ are zero. 
Clearly, we have $\sum_{k=1}^n |x^*_k(a)| \leq 1$ and $\sum_{k=1}^n |y^*_k(a)| \leq 1$ for all $a\in \mathbb{N}$,
so Corollary~\ref{c:FBL(A)}(i) yields
$$
	\|f-g\|_{FBL[\ell_1]}\geq \max\left(\sum_{k=1}^n |f(x_k^*)-g(x_k^*)|, \,\sum_{k=1}^n |f(y_k^*)-g(y_k^*)|\right).
$$
But $g(x^*_k) = g(y^*_k)$ for every $k\in \{1,\dots,n\}$, because 
$g$ belongs to the sublattice generated by $\{\delta_{e_1},\dots,\delta_{e_n}\}$ and $x_k^*(a)=y_k^*(a)$ for every $a\in \{1,\dots,n\}$. Hence
$$
	\|f-g\|_{FBL[\ell_1]}\geq \frac{1}{2} \sum_{k=1}^n |f(x_k^*)-f(y_k^*)|.
$$
By the definition of~$f$, we have $f(x^*_k)=\frac{1}{n+k}$ and $f(y^*_k) = 0$ for every $k\in\{1,\dots,n\}$, so
$$
	\|f-g\|_{FBL[\ell_1]}\geq \frac{1}{2} \sum_{k=1}^n \frac{1}{n+k}\geq 
	\frac{1}{4},
$$
as we wanted to show. \qed
\end{example}

Let $\mathfrak P(B_{E^{**}})$ denote the set of regular Borel
probabilities on $(B_{E^{**}},w^*)$. Then $\mathfrak P(B_{E^{**}})$ is a convex $w^*$-compact
subset of the dual of $C(B_{E^{**}},w^*)$. Note that each $\mu\in \mathfrak P(B_{E^{**}})$ induces a function $f_\mu:E^*\to \erre_+$ by 
$$
	f_\mu(x^*):=\int_{B_{E^{**}}} |x^*(\cdot)| \, d\mu
	\quad
	\mbox{for all }x^*\in E^*.
$$
This provides a link between $H_0[E]_+$ and $\mathfrak P(B_{E^{**}})$, as we next explain.

\begin{prop}\label{prop:fmu}
If $\mu\in \mathfrak P(B_{E^{**}})$, then
$f_\mu\in H_0[E]_+$ and $\|f_\mu\|_{FBL[E]}\leq 1$.
\end{prop}
\begin{proof} Clearly, $f_\mu$ is positively homogeneous. 
Given $x_1^*,\ldots, x_n^*\in E^*$ we have
$$
	\sum_{i=1}^n f_\mu (x^*_i)=  \int_{B_{E^{**}}} \sum_{i=1}^n |x_i^*(\cdot)| \, d\mu
	\leq \sup_{x^{**}\in B_{E^{**}}}\sum_{i=1}^n|x_i^*(x^{**})|= \sup_{x\in B_{E}}\sum_{i=1}^n|x_i^*(x)|,
$$
where the last equality holds because $B_E$ is $w^*$-dense in~$B_{E^{**}}$.
It follows that $\|f_\mu\|_{FBL[E]} \leq 1$.
\end{proof}

\begin{prop}\label{l:kyfan}
For every $f\in H_0[E]_+$ there is $\mu \in \mathfrak P(B_{E^{**}})$ such that 
$$
	f(x^*) \leq \|f\|_{FBL[E]} \, f_\mu(x^*)
	\quad \mbox{for all }x^*\in E^*.
$$
\end{prop}
\begin{proof}
Assume without loss of generality that $\|f\|_{FBL[E]}=1$. Given any finite collection $x_1^*,\ldots, x_n^*\in E^*$, let
$$
\begin{array}{cccl}
	\varphi_{x_1^*,\ldots, x_n^*}:&\mathfrak P(B_{E^{**}})&\longrightarrow&\mathbb R\\
	&\mu&\longmapsto& \displaystyle \sum_{i=1}^n \Big(f(x_i^*)-\int_{B_{E^{**}}}|x_i^*(\cdot)| \, d\mu\Big).
\end{array}
$$
It is clear that the function $\varphi_{x_1^*,\ldots, x_n^*}$ is convex and $w^*$-continuous. Moreover,
since $\|f\|_{FBL[E]}=1$ and $B_E$ is $w^*$-dense in~$B_{E^{**}}$, we have
$$
	\sum_{i=1}^n f(x_i^*)\leq \sup_{x\in B_E}\sum_{i=1}^n|x_i^*(x)|=
	\sup_{x^{**}\in B_{E^{**}}} \sum_{i=1}^n|x_i^*(x^{**})|.
$$
The last supremum is attained at some $x^{**}_0\in B_{E^{**}}$ because
$(B_{E^{**}},w^*)$ is compact and
the map $\sum_{i=1}^n|x_i^*(\cdot)|$ is $w^*$-continuous.
This means that the value of $\varphi_{x_1^*,\ldots, x_n^*}$ at the probability measure concentrated on~$x^{**}_0$ is less
than or igual to~$0$.

By using that $f$ is positively homogeneous, it is easy to check that the collection of all functions of the form $\varphi_{x_1^*,\ldots, x_n^*}$ 
is a convex cone of~$\erre^{\mathfrak P(B_{E^{**}})}$.
From Ky Fan's lemma (see e.g. \cite[9.10]{DJT}) it follows that there is $\mu\in \mathfrak P(B_{E^{**}})$ such that 
$\varphi_{x_1^*,\ldots, x_n^*}(\mu)\leq0$ for every $n\in \Nat$ and $x_1^*,\ldots, x_n^*\in  E^*$. In particular, this yields that for every $x^*\in E^*$ 
we have $\varphi_{x^*}(\mu)\leq 0$, that is,
$$
	f(x^*)\leq \int_{B_{E^{**}}}|x^*(\cdot)|\, d\mu=f_\mu(x^*).
$$
The proof is complete.
\end{proof}

\section{Density character of order intervals}\label{s:OrderIntervals}

Recall that the {\em density character} of a topological space~$T$, denoted by~${\rm dens}(T)$, is the 
least cardinality of a dense subset. Given any Banach space~$E$, we have 
$$
	{\rm dens}(E)={\rm dens}(FBL[E])
$$
since the sublattice generated by~$\phi_E(E)=\{\delta_x:x\in E\}$ is dense in $FBL[E]$.
In \cite[Question 12.5]{dePW}, the authors ask 
whether every non-degenerate order interval in $FBL(A)$ 
(for an arbitrary non-empty set~$A$) has the same density character. Theorem~\ref{t:orderintervals} below shows that this is indeed the case, even in the more general setting of $FBL[E]$.
The proof requires the following lemma, which might be known.

\begin{lem}\label{l:separation}
Let $E$ be an infinite-dimensional Banach space with ${\rm dens}(E)=\kappa$.
Then there exist $\delta>0$ and a linearly independent set $S \sub S_E$ with $|S|=\kappa$ 
such that $\|x-\lambda x'\|\geq\delta$ for every distinct $x,x'\in S$
and every $\lambda\in\mathbb R$.
\end{lem}
\begin{proof}
Fix $x^*\in S_{E^*}$ and $0<\epsilon<1$. Define $S_0:=\{x\in S_E: x^*(x) >\epsilon\}$.
We will first check that $S_0$ is infinite. To this end, note that
$$
	W := \{x \in E :\, \|x\|<1,\ x^*(x)>\epsilon\}
$$ 
is open and non-empty, hence 
there exist $x_0\in W$ and $r>0$ in such a way that $x_0+rB_E \sub W$.
Since $T := \{x \in x_0+rB_E : x^*(x)=x^*(x_0)\}$ is infinite (because it is a closed ball of the affine hyperplane $x_0 + \ker(x^*)$) and 
the map $h: T \to S_0$ given by $h(x) := \| x \|^{-1} \, x$ is one-to-one, we conclude that $S_0$ is infinite.

We next show that ${\rm dens}(S_0)=\kappa$. Indeed, take any dense set $D \sub S_0$.
Then $D$ is infinite and so it has the same cardinality as $\tilde{D}:=\{\lambda x: x\in D, \lambda\in \mathbb{Q}, \lambda>0\}$, which
is dense in
$$
	G:=\{\lambda x: \, x\in S_0, \, \lambda\in \mathbb{R}, \, \lambda>0\}=\{x\in E: \, x^*(x)>\epsilon\|x\|\}.
$$
Since $G$ is open and non-empty, we have $\kappa={\rm dens}(G)\leq |\tilde{D}|=|D|$.
This proves that ${\rm dens}(S_0)=\kappa$.

So, there exist $0<\delta<\epsilon$ and a set $S_1 \sub S_0$ with $|S_1|=\kappa$ such that
$\|x-x'\|\geq 2\delta$ for every distinct $x,x'\in S_1$. Let $S \sub S_1$ be a maximal linearly independent subset.
We shall check that $S$ satisfies the required properties. By maximality, we have $S_1 \sub {\rm span}(S)$
and therefore
\begin{enumerate}
\item[(i)] $S$ is infinite (bear in mind that the closed unit ball of a finite-dimensional Banach space is compact);
\item[(ii)] ${\rm dens}({\rm span}(S))=\kappa$. 
\end{enumerate}
From (i) and (ii) it follows that $|S|=\kappa$. Now, fix $x,x'\in S$ with $x\neq x'$
and take any $\lambda \in \erre$. We next show that $\|x-\lambda x'\| \geq \delta$
by considering several cases:
\begin{itemize}
\item If $|1+\lambda| < \delta$, then $\|(x-\lambda x')-(x+x')\|<\delta$
and so 
$$
	\|x-\lambda x'\| > \|x+x'\|-\delta\geq x^*(x+x')-\delta>2\epsilon-\delta>\delta. 
$$
\item If $|1-\lambda| < \delta$, then
$$
	\|x-\lambda x'\| \geq \|x-x'\| - \|x'-\lambda x'\|>\|x-x'\|-\delta \geq \delta.
$$
\item If $|\lambda|\leq 1-\delta$, then
$$
	\|x-\lambda x'\| \geq \|x\|-|\lambda|\|x'\|=1-|\lambda| \geq \delta.
$$
\item If $|\lambda|\geq 1+\delta$, then 
$$
	\|x-\lambda x'\| \geq |\lambda|\|x'\|-\|x\|=|\lambda|-1 \geq \delta.
$$
\end{itemize} 
The proof is finished.
\end{proof}

\begin{thm}\label{t:orderintervals}
Let $E$ be a Banach space. Then every non-degenerate order interval in $FBL[E]$ has the same density character as~$E$.
\end{thm}
\begin{proof}
If $E$ is separable, then so is $FBL[E]$ and, in particular, every order interval in $FBL[E]$ is separable. 

Let us assume now that $E$ is non-separable with ${\rm dens}(E)=\kappa$. 
Pick $f\leq g$ in $FBL[E]$ with $f\neq g$. 
Rescaling, we can suppose that $f$ and $g$ belong to~$B_{FBL[E]}$.
There is a {\em separable} closed subspace $F\sub E$ such that $f$ and $g$
belong to the closed sublattice of~$FBL[E]$ generated by $\{\delta_{x}:x \in F\}$.
Bearing in mind that ${\rm dens}(E/F)=\kappa$, we can apply Lemma~\ref{l:separation} to~$E/F$ in order to find $\delta>0$ and 
a set $S \sub S_{E/F}$ with $|S|=\kappa$, consisting of linearly independent vectors, such that
$\|s-\lambda t\|\geq\delta$ for every distinct $s,t \in S$ and every $\lambda\in\mathbb R$. 

Let $\pi:E\rightarrow E/F$ denote the quotient operator. For each $s\in S$, we take $u_s\in \pi^{-1}(s)$ with $\|u_s\|\leq 2$ and we define 
$$
	z_s := (\delta_{u_s}\vee f)\wedge g\in [f,g] \sub FBL[E].
$$ 
To finish the proof it is enough to prove that there is a constant $c>0$ such that $\|z_s - z_t\| \geq c$ for every $s,t \in S$ with $s\neq t$.
Fix $x^*\in B_{E^*}$ such that $f(x^*)<g(x^*)$ and consider the evaluation functional $\varphi_{x^*}\in B_{FBL[E]^*}$
given by~$\varphi_{x^*}(h):=h(x^*)$ for all $h\in FBL[E]$ (see Corollary~\ref{cor:RealHomomorphisms}).
Set $C:=6\|\pi\|/\delta+1$ and fix $s\neq t$ in~$S$.

{\em Claim.} The inequality 
\begin{equation}\label{eqn:Pedro}
	\|x\|+|\alpha|+|\beta|\leq C\|x+\alpha u_s+\beta u_t\|
\end{equation}
holds for every $x\in F$ and every $\alpha,\beta \in\mathbb R$. Indeed, by the choice of~$S$ we have
$$
	\|\pi\|\, \|x+\alpha u_s+\beta u_t\|\geq \|\alpha s+\beta t\| \geq \delta \, \max\{|\alpha|,|\beta|\} \geq \delta\, \frac{|\alpha|+|\beta|}{2}.
$$
On the other hand, since $\|u_s\|\leq 2$ and $\|u_t\|\leq 2$, we have
$$
	\|x+\alpha u_s+\beta u_t\| \geq \|x\|-\|\alpha u_s+\beta u_t\| \geq \|x\|-2\big(|\alpha|+|\beta|\big).
$$
Therefore
\begin{multline*}
	C \|x+\alpha u_s+\beta u_t\|=\frac{6\|\pi\|}{\delta}  \|x+\alpha u_s+\beta u_t\| + \|x+\alpha u_s+\beta u_t\| \\
	\geq 3\big(|\alpha|+|\beta|\big)+\|x\|-2\big(|\alpha|+|\beta|\big)=\|x\|+|\alpha|+|\beta|,
\end{multline*}
which finishes the proof of the claim.
 
Note that $u_s$ and $u_t$ are linearly independent vectors in~$E\setminus F$. Let 
$$
	\xi: {\rm span}(F\cup\{u_s,u_t\})\rightarrow \mathbb R
$$ 
be the linear functional given by 
\begin{multline*}
	\xi(x+\alpha u_s+\beta u_t):=\varphi_{x^*}(\delta_x+\alpha f+\beta g) \\
	=x^*(x)+\alpha f(x^*)+\beta g(x^*)
	\quad
	\mbox{for all }x\in F \mbox{ and } \alpha,\beta\in \erre.
\end{multline*}
Bearing in mind~\eqref{eqn:Pedro}, we get 
$$
	|\xi(x+\alpha u_s+\beta u_t)|\leq \| \delta_x+\alpha f+\beta g\|_{FBL[E]} \leq  \|x\|+|\alpha|+|\beta|\leq C\|x+\alpha u_s+\beta u_t\|
$$ 
for every $x\in F$ and $\alpha,\beta\in \erre$. By the Hahn-Banach theorem, $\xi$ can be extended to an element of~$E^*$, still denoted by~$\xi$, 
with $\|\xi\|\leq C$. 

Now, let $\hat{\xi}\in FBL[E]^*$ be the lattice homomorphism satisfying $\hat \xi \circ \phi_E=\xi$ and $\|\hat{\xi}\|=\|\xi\|\leq C$
(Theorem~\ref{t:fblb}). Note that 
$$
	\hat{\xi}(\delta_x)=\xi(x)=\varphi_{x^*}(\delta_x)=x^*(x)\quad
	\mbox{for every }x\in F.
$$  
Since $\hat\xi$ and $\varphi_{x^*}$
are lattice homomorphisms and $f$ and~$g$ belong to the closed sublattice generated by $\{\delta_x:x\in F\}$, 
it follows that $\hat{\xi}(f)=\varphi_{x^*}(f)=f(x^*)$ and $\hat{\xi}(g)=\varphi_{x^*}(g)=g(x^*)$. In particular, we have
$$
	C\|z_s - z_t\| \geq\Big|\hat{\xi}\Big((\delta_{u_s}\vee f)\wedge g-(\delta_{u_t}\vee f)\wedge g\Big)\Big|=g(x^*)-f(x^*).
$$
This shows that $\|z_s - z_t\|\geq C^{-1}(g(x^*)-f(x^*))$ for every $s\neq t$ in~$S$. The proof is finished.
\end{proof}

\section{The Nakano property}\label{s:nakano}

The norm of a Banach lattice $X$ is said to have the \emph{Nakano property} if for every upwards directed order bounded set 
$\mathcal{F}\sub X_+$ we have
$$
	\sup\{\|x\|: \, x\in\mathcal F\}=\inf\{\|y\|: \, y \in X \mbox{ is an upper bound of } \mathcal{F}\}.
$$
This property was introduced in \cite{N} (see also \cite{W}) and is stronger than the \emph{Fatou property}, which 
simply states that whenever an upwards directed set $\mathcal{F}\sub X_+$ has a supremum $y \in X$, then
$$
	\sup\{\|x\|:\, x\in\mathcal F\}=\|y\|.
$$

In \cite[Question 12.1]{dePW}, it was asked whether the norm of $FBL(A)$ has the Nakano property. We will show 
in Theorem~\ref{thm:Nakano} that this is the case. In fact, a stronger property holds:

\begin{defn}
We say that the norm of a Banach lattice $X$ has the \emph{strong Nakano property} if for every upwards directed norm bounded 
set $\mathcal{F}\sub X_{+}$ there exists an upper bound $y_0$ of~$\mathcal{F}$ in~$X$ such that 
$$
	\sup\{\|x\|: \, x\in\mathcal F\}=\|y_0\|.
$$	
\end{defn}

The supremum norm of a $C(K)$ space ($K$ being a compact Hausdorff topological space) has the strong Nakano property, 
because we can take $y_0$ as the constant function equal to $\sup\{\|x\|_\infty :x\in\mathcal F\}$. In a sense, we shall see that the free 
Banach lattices $FBL(A)=FBL[\ell_1(A)]$ have an analogous structure, the role of the positive constant functions 
being played by the elements of the form $|\delta_{e_a}|$ for $a\in A$. Our proof of Theorem~\ref{thm:Nakano} requires some preliminary lemmas. 

\begin{lem}\label{lem:PointwiseSupremum}
Let $E$ be a Banach space and let $\mathcal{F}\sub H[E]_+$ be upwards directed and pointwise bounded. Define $g: E^*\to \erre_+$ 
by $g(x^*):=\sup\{f(x^*):f\in \mathcal{F}\}$ for all $x^*\in E^*$. Then $g\in H[E]_+$ and 
$$
	\|g\|_{FBL[E]} = \sup \{\|f\|_{FBL[E]} : \, f\in\mathcal{F}\}.
$$ 
\end{lem}
\begin{proof}
Clearly, $g$ is positively homogeneous and $\|g\|_{FBL[E]}  \geq \|f\|_{FBL[E]}$ for every $f\in\mathcal{F}$. 
To prove that $\|g\|_{FBL[E]} \leq \sup \{\|f\|_{FBL[E]} : \, f\in\mathcal{F}\}:=\alpha$ we can assume that
the supremum is finite. Fix $\epsilon>0$. Take any $x_1^*,\dots,x_n^* \in E^*$ such that $\sum_{k=1}^n |x_k^*(x)|\leq 1$
for every $x\in B_E$. Since $\mathcal{F}$ is upwards directed, we can find $f\in \mathcal{F}$ such that
$g(x_k^*)-\frac{\epsilon}{n} \leq f(x_k^*)$ for all $k\in \{1\dots,n\}$, therefore
$$
	\sum_{k=1}^n g(x^*_k) \leq \epsilon+\sum_{k=1}^n f(x^*_k) \leq \epsilon+\alpha.
$$
It follows that $\|g\|_{FBL[E]} \leq \epsilon+\alpha$. As $\epsilon>0$ is arbitrary, $\|g\|_{FBL[E]} \leq \alpha$.
\end{proof}

\begin{defn}
Let $E$ be a Banach space. We say that $f\in H_0[E]_+$ is {\em maximal} if 
$$
	\big\{g\in H_0[E]_+: \, g \geq f, \, \|g\|_{FBL[E]} = \|f\|_{FBL[E]} \big\}=\{f\}.
$$
\end{defn}

\begin{lem}\label{lem:Zorn}
Let $E$ be a Banach space and $f\in H_0[E]_+$. Then there exists a maximal $\tilde{f} \in H_0[E]_+$ such that 
$f\leq \tilde{f}$  and $\|f\|_{FBL[E]} = \|\tilde{f}\|_{FBL[E]}$.
\end{lem}
\begin{proof}
The set $\mathcal{G}:=\{g\in H_0[E]_+: \, g\geq f, \, \|g\|_{FBL[E]}=\|f\|_{FBL[E]}\}$ 
is pointwise bounded by Remark~\ref{rem:PointwiseBounded}. Note that every 
upwards directed subset of~$\mathcal{G}$ has an upper bound in~$\mathcal{G}$
(by Lemma~\ref{lem:PointwiseSupremum}). Thus,
Zorn's lemma ensures the existence of an element of~$\mathcal{G}$ which is maximal for the pointwise ordering.
\end{proof}

Our next aim is to identify the maximal elements in $H_0[\ell_1(A)]_+$ for an arbitrary non-empty set~$A$.
We shall use without explicit mention the formula to compute the norm $\|\cdot\|_{FBL[\ell_1(A)]}$ 
given in Corollary~\ref{c:FBL(A)}(i).

\begin{lem}\label{lem:PropertiesMaximal}
Let $A$ be a non-empty set and let $f\in H_0[\ell_1(A)]_+$ be maximal. 
\begin{enumerate}
\item[(i)] $f(x^*)\leq f(y^*)$ whenever $x^*,y^*\in \ell_\infty(A)$ satisfy $|x^*|\leq |y^*|$.
\item[(ii)]  $\sum_{k=1}^n f(x^*_k) \leq f(\sum_{k=1}^n x^*_k)$
for every $n\in\mathbb N$ and $x_1^*,\dots,x_n^*\in \ell_\infty(A)_+$.
\item[(iii)] $\|f\|_{FBL[\ell_1(A)]} = \|f\|_\infty$.
\end{enumerate}
\end{lem}
\begin{proof} For any $z^*\in \ell_\infty(A)$ we write $R(z^*):=\{\lambda z^*:  \lambda>0\}\sub \ell_\infty(A)$.

(i): By contradiction, suppose that $f(x^*)>f(y^*)$. 
Define $g:\ell_\infty(A)\to \erre_+$ by   
$$
	\begin{cases}
		g(\lambda y^*):=f(\lambda x^*) & \text{for all $\lambda>0$}, \\
		g(z^*):=f(z^*) & \text{for all $z^*\in \ell_\infty(A) \setminus R(y^*)$}.
	\end{cases}
$$
It is easy to check that $g$ is positively homogeneous. Since $f(x^*)>f(y^*)$, we have $f\leq g$ and $f\neq g$.
Bearing in mind that $f$ is maximal, in order to get a contradiction it suffices to check that $\|g\|_{FBL[\ell_1(A)]}=\|f\|_{FBL[\ell_1(A)]}$.
To this end, take any $x_1^*,\dots,x_n^*\in \ell_\infty(A)$ such that $\sum_{k=1}^n |x_k^*(a)|\leq 1$ for all $a\in A$. Let
$I$ be the set of those $k\in \{1,\dots,n\}$ such that $x_k^*\not\in R(y^*)$ and let $J:=\{1,\dots,n\}\setminus I$, so that for each $k\in J$
we have $x_k^*=\lambda_k y^*$ for some $\lambda_k>0$. Note that
\begin{multline*}
	\sum_{k\in I} |x_k^*(a)|+\sum_{k\in J}  |\lambda_k x^*(a)|=
	\sum_{k\in I} |x_k^*(a)|+|x^*(a)|\sum_{k\in J}  \lambda_k \\ \leq 
	\sum_{k\in I} |x_k^*(a)|+|y^*(a)|\sum_{k\in J}  \lambda_k=
	\sum_{k=1}^n |x_k^*(a)| \leq 1 \quad\mbox{for all }a\in A,
\end{multline*}
hence
$$
	\sum_{k=1}^n g(x_k^*)=
	\sum_{k\in I} f(x_k^*)+\sum_{k\in J}  f(\lambda_k x^*)
	\leq \|f\|_{FBL[\ell_1(A)]}.
$$
This shows that $\|g\|_{FBL[\ell_1(A)]} \leq \|f\|_{FBL[\ell_1(A)]}$, a contradiction.

(ii): Set $x^*:=\sum_{k=1}^n x^*_k$. By contradiction, suppose that $f(x^*)<\sum_{k=1}^n f(x_k^*)$. 
Define $g:\ell_\infty(A)\to \erre_+$ by   
$$
	\begin{cases}
		g(\lambda x^*):=\lambda \sum_{k=1}^n f(x_k^*) & \text{for all $\lambda>0$}, \\
		g(z^*):=f(z^*) & \text{for all $z^*\in \ell_\infty(A) \setminus R(x^*)$}.
	\end{cases}
$$
Clearly, $g$ is positively homogeneous, $f\leq g$ and $f\neq g$. Again by the maximality of~$f$,
to get a contradiction it suffices to show that $\|g\|_{FBL[\ell_1(A)]}=\|f\|_{FBL[\ell_1(A)]}$.
Take $y_1^*,\dots,y_m^*\in \ell_\infty(A)$ such that $\sum_{j=1}^m |y_j^*(a)|\leq 1$ for all $a\in A$. Let
$I$ denote the set of all $j\in \{1,\dots,m\}$ for which $y_j^*\not\in R(x^*)$ and let $J:=\{1,\dots,m\}\setminus I$, so that for each $j\in J$
we can write $y_j^*=\lambda_j x^*$ for some $\lambda_j>0$. Set $\mu:=\sum_{j\in J}\lambda_j$.
Since
\begin{multline*}
	\sum_{j\in I}|y_j^*(a)|+\sum_{k=1}^n|\mu x_k^*(a)|=
	\sum_{j\in I}|y_j^*(a)|+\mu x^*(a) \\ =
	\sum_{j\in I}|y_j^*(a)|+\sum_{j\in J} \lambda_j x^*(a)
	= \sum_{j=1}^m|y_j^*(a)| \leq 1 \quad\mbox{for all }a\in A,
\end{multline*}
we obtain
\begin{multline*}
	\sum_{j=1}^m g(y_j^*)
	=\sum_{j\in I}f(y_j^*) + \sum_{j\in J}\lambda_j\Big(\sum_{k=1}^n f(x_k^*)\Big)
	\\ =\sum_{j\in I} f(y_j^*)+\sum_{k=1}^n f(\mu x_k^*) \leq \|f\|_{FBL[\ell_1(A)]}.
\end{multline*}
It follows that $\|g\|_{FBL[\ell_1(A)]} \leq \|f\|_{FBL[\ell_1(A)]}$, which is a contradiction.

(iii): By Remark~\ref{rem:PointwiseBounded} we have $\|f\|_{FBL[\ell_1(A)]} \geq \|f\|_\infty$. To prove the
equality, take finitely many $x_1^*,\dots,x_n^*\in \ell_\infty(A)$
such that $\sum_{k=1}^n |x_k^*(a)|\leq 1$ for every $a\in A$. Then $x^*:=\sum_{k=1}^n |x_k^*|\in B_{\ell_\infty(A)}$ and
$$
	\sum_{k=1}^n f(x^*_k) \stackrel{{\rm (i)}}{\leq}
	\sum_{k=1}^n f(|x^*_k|) \stackrel{{\rm (ii)}}{\leq} f(x^*) \leq \|f\|_\infty.
$$
This shows that $\|f\|_{FBL[\ell_1(A)]} \leq \|f\|_\infty$ and finishes the proof.
\end{proof}

\begin{lem}\label{lem:normoflinear}
Let $A$ be a non-empty set and let $\phi:\ell_\infty(A) \to \mathbb{R}$ be a linear functional. Define $g_\phi:\ell_\infty(A) \to \erre_+$ by
$$
	g_\phi(x^*) := |\phi(|x^*|)|
	\quad\mbox{ for all }x^*\in \ell_\infty(A). 
$$
Then $g_\phi\in H[\ell_1(A)]_+$ and 
$$
	\|g_\phi\|_{FBL[\ell_1(A)]} = \sup\big\{|\phi(x^*)|: \, x^*\in B_{\ell_\infty(A)}\big\}.
$$
\end{lem}
\begin{proof}
Clearly, $g_\phi$ is positively homogeneous. Take any $x^*_1,\ldots,x^*_n\in \ell_\infty(A)$ such that 
$\sum_{k=1}^n|x^*_k(a)|\leq 1$ for all $a\in A$. For each $k\in \{1,\dots,n\}$, let $\varepsilon_k \in \{-1,1\}$ be the sign of $\phi(|x^*_k|)$.
Then $\sum_{k=1}^n  \varepsilon_k|x^*_k| \in B_{\ell_\infty(A)}$ and so
$$
	\sum_{k=1}^n g_\phi(x^*_k) = \sum_{k=1}^n \varepsilon_k \phi(|x^*_k|) 
	= \phi\left(\sum_{k=1}^n  \varepsilon_k|x^*_k|\right) \leq \sup\big\{|\phi(x^*)|: \, x^*\in B_{\ell_\infty(A)}\big\}=:\alpha.
$$
This immediately shows that $\|g_\phi\|_{FBL[\ell_1(A)]}\leq \alpha$. 

For the converse, pick $x^*\in B_{\ell_\infty(A)}$ and write it as $x^* = (x^*)^+ - (x^*)^-$, 
the difference of its positive and negative parts. Since $|(x^*)^+(a)| + |(x^*)^-(a)|=|x^*(a)|\leq 1$ for all $a\in A$, we have
$$
	|\phi(x^*)| \leq |\phi((x^*)^+)| + |\phi((x^*)^-)| = g_\phi((x^*)^+) + g_\phi((x^*)^-) \leq \|g_\phi\|_{FBL[\ell_1(A)]}.
$$ 
This proves that $\alpha\leq \|g_\phi\|_{FBL[\ell_1(A)]}$.
\end{proof}

\begin{lem}\label{lem:MaximalNecessaryCondition}
Let $A$ be a non-empty set and let $f \in H_0[\ell_1(A)]_+$ be maximal. 
Then there exists $\phi\in\ell_\infty(A)^*$ such that $f=g_\phi$.
\end{lem}
\begin{proof}
The case $f=0$ being trivial, we can suppose without loss of generality that $\|f\|_{FBL[\ell_1(A)]} = 1$. 
The set
$$
	C:= \{x^*\in \ell_\infty(A)_+ :\, f(x^*)>1\}
$$ 
is convex as a consequence of Lemma~\ref{lem:PropertiesMaximal}(ii). Let $U$ be the open
unit ball of~$\ell_\infty(A)$. Since $\|f\|_\infty = \|f\|_{FBL[\ell_1(A)]} = 1$
(Lemma~\ref{lem:PropertiesMaximal}(iii)), we have $C\cap U=\emptyset$.
As an application of the Hahn-Banach separation theorem (cf. \cite[Proposition~2.13(ii)]{fab-ultimo}),
there is $\phi\in \ell_\infty(A)^*$ such that
\begin{equation}\label{eqn:HB}
	\phi(y^*) < \inf \{\phi(x^*) :\,  x^* \in C\}
	\quad\mbox{for all }y^*\in U.
\end{equation}
We can suppose that $\|\phi\|=1$ and so $\|g_\phi\|_{FBL[\ell_1(A)]}=1$ (Lemma~\ref{lem:normoflinear}). 

We claim that $f=g_\phi$. Indeed, since $f$ is maximal, it suffices to prove that $f(x^*)\leq g_\phi(x^*)$ for every
$x^*\in \ell_\infty(A)$ with $f(x^*)>0$. Fix $t>1$. By Lemma~\ref{lem:PropertiesMaximal}(i), we have
$f(|x^*|)\geq f(x^*)>0$ and so
$$
	f\left(\frac{t}{f(|x^*|)}|x^*|\right) = t > 1.
$$ 
Therefore, \eqref{eqn:HB} yields
$$
	\phi\left(\frac{t}{f(|x^*|)}|x^*|\right) \geq \sup \{|\phi(y^*)|:\, y^*\in U\}=\|\phi\|=1.
$$
We conclude that $t\phi(|x^*|) \geq f(|x^*|) $ for any $t>1$, so $g_\phi(|x^*|)=\phi(|x^*|)\geq f(x^*)$.
The proof is complete.
\end{proof}

Given any non-empty set~$A$, it is well-known that every $\phi\in \ell_\infty(A)^*$
can be written in a unique way as $\phi = \phi_0 + \phi_1$, where 
\begin{itemize}
\item $\phi_0 \in \ell_1(A)$ (identified as a subspace of~$\ell_\infty(A)^*$), 
\item $\phi_1 \in \ell_\infty(A)^*$ vanishes on all finitely supported elements of~$\ell_\infty(A)$.
\end{itemize}
Moreover, $\|\phi\| = \|\phi_0\| + \|\phi_1\|$. 

\begin{lem}\label{lem:CountableSum}
Let $A$ be a non-empty set and $\phi\in \ell_1(A)$. Then $g_\phi \in FBL[\ell_1(A)]$.
\end{lem}
\begin{proof} Let $(e_a)_{a\in A}$ be the unit vector basis of~$\ell_1(A)$.
The series $\sum_{a\in A} \phi(a) |\delta_{e_a}|$ is summable in $FBL[\ell_1(A)]$
because $\phi\in \ell_1(A)$ and $\|\delta_{e_a}\|_{FBL[\ell_1(A)]}=1$ for every $a\in A$.
Let $h \in FBL[\ell_1(A)]$ be its sum. By Remark~\ref{rem:PointwiseBounded}, we have
$$
	h(x^*) = \sum_{a\in A} \phi(a) |x^*(a)| 
	\quad\mbox{for all }x^*\in \ell_\infty(A).
$$
Therefore, $|h(x^*)|=g_\phi(x^*)$ for all $x^*\in \ell_\infty(A)$
and so $g_\phi\in FBL[\ell_1(A)]$.
\end{proof}

\begin{lem}\label{lem:continuity}
Let $A$ be a non-empty set, $\xi: B_{\ell_\infty(A)} \to \erre$ a $w^*$-continuous function and $\phi\in \ell_\infty(A)^*$. 
If $\xi \leq g_\phi$ on~$B_{\ell_\infty(A)}$, then $\xi \leq g_{\phi_0}$ on~$B_{\ell_\infty(A)}$ as well.
\end{lem}
\begin{proof} The $w^*$-topology on~$B_{\ell_\infty(A)}=[-1,1]^A$ agrees with the pointwise topology.
Since the map $x^* \mapsto |x^*|$ is $w^*$-$w^*$-continuous when restricted to~$B_{\ell_\infty(A)}$ and $\phi_0$ is $w^*$-continuous, 
we have that $g_{\phi_0}$ is $w^*$-continuous on~$B_{\ell_\infty(A)}$.
On the other hand, if $x^*\in B_{\ell_\infty(A)}$ is finitely supported, then $\phi_1(|x^*|) = 0$ and, therefore, we have $\xi(x^*) \leq g_{\phi}(x^*) 
= |\phi_0(|x^*|)|=g_{\phi_0}(x^*)$. 
Since the finitely supported elements of~$B_{\ell_\infty(A)}$ are $w^*$-dense and the functions $\xi$ and $g_{\phi_0}$ are $w^*$-continuous on~$B_{\ell_\infty(A)}$, 
we conclude that $\xi \leq g_{\phi_0}$ on~$B_{\ell_\infty(A)}$.
\end{proof}

\begin{lem}\label{lem:continuity2}
Let $E$ be a Banach space. Then for every $f\in FBL[E]$ the restriction $f|_{B_{E^{*}}}$ is $w^*$-continuous.
\end{lem}
\begin{proof}
The set $S$ consisting of all $f\in FBL[E]$ for which $f|_{B_{E^{*}}}$ is $w^*$-continuous
is clearly a sublattice of~$FBL[E]$ containing $\{\delta_x:x\in E\}$. Moreover, $S$ is closed in $FBL[E]$ (by Remark~\ref{rem:PointwiseBounded}), 
so $FBL[E]=S$.
\end{proof}

We arrive at the main result of this section:

\begin{thm}\label{thm:Nakano}
The norm of $FBL[\ell_1(A)]$ has the strong Nakano property for any non-empty set~$A$.
\end{thm}
\begin{proof}
Let $\mathcal{F} \sub FBL[\ell_1(A)]_+$ be an upwards directed family such that 
$$
	\sup\{\|f\|_{FBL[\ell_1(A)]} : \, f\in\mathcal{F}\} = 1.
$$ 
We are going to show that $\mathcal{F}$ has an upper bound of norm 1. 

Note that $\mathcal{F}$ is pointwise bounded (Remark~\ref{rem:PointwiseBounded}) and let $h:\ell_\infty(A)\to \erre_+$ be defined
as $h(x^*):=\sup\{f(x^*):f\in \mathcal{F}\}$ for all $x^*\in \ell_\infty(A)$. 
Lemma~\ref{lem:PointwiseSupremum} ensures that $h\in H_0[\ell_1(A)]_+$ and $\|h\|_{FBL[\ell_1(A)]} = 1$. 
Now let $g\in H_0[\ell_1(A)]_+$ be maximal such that 
$g\geq h$  and $\|g\|_{FBL[\ell_1(A)]}=1$ (apply Lemma~\ref{lem:Zorn}).
Then $g=g_\phi$ for some $\phi \in \ell_\infty(A)^*$ with $\|\phi\|=1$
(combine Lemmas~\ref{lem:MaximalNecessaryCondition} and~\ref{lem:normoflinear}).

Given any $f\in \mathcal{F} \sub FBL[\ell_1(A)]$, we have $f\leq g_\phi$ and the restriction $f|_{B_{\ell_\infty(A)}}$ is $w^*$-continuous 
(Lemma~\ref{lem:continuity2}), hence Lemma~\ref{lem:continuity} yields $f(x^*)\leq g_{\phi_0}(x^*)$ for every $x^*\in \ell_\infty(A)$ (bear in mind that
both $f$ and $g_{\phi_0}$ are positively homogeneous). Since $g_{\phi_0} \in FBL[\ell_1(A)]$
(Lemma~\ref{lem:CountableSum}) and $1=\|\phi\| \geq \|\phi_0\|=\|g_{\phi_0}\|_{FBL[\ell_1(A)]}$
(Lemma~\ref{lem:normoflinear}), it turns out that $g_{\phi_0}$ is the upper bound of~$\mathcal{F}$ in $FBL[\ell_1(A)]$ that we were looking for. 
The proof is finished.
\end{proof}

It is natural to wonder whether the norm of $FBL[E]$ also has the (strong) Nakano property for an arbitrary Banach space~$E$. 
We will next show that the norm of $FBL[L_1]$ is not even Fatou, where $L_1$ denotes the space $L_1([0,1],\mu)$ and
$\mu$ is the Lebesgue measure. 

The following auxiliary lemma belongs to the folklore and we include its proof for the sake of completeness.

\begin{lem}\label{l:fn}
For each $n\in\mathbb N$, define $f_n: L_1 \to \erre_+$ by
$$
	f_n(h):=\sum_{j=1}^{2^n} \Big|\int_{I_{n,j}} h \, d\mu\Big|
	\quad \mbox{for every }h\in L_1,
$$
where $I_{n,j}:=[\frac{j-1}{2^n},\frac{j}{2^n}]$ for all $j\in \{1,\dots,2^n\}$.
The following properties hold:
\begin{enumerate}
\item[(i)] $f_n(h)\leq f_{n+1}(h)$ for every $n\in\mathbb N$ and $h\in L_1$.
\item[(ii)] $\lim_{n\to \infty} f_n(h)=\int_0^1|h| \, d\mu$ for every $h\in L_1$.
\end{enumerate}
\end{lem}
\begin{proof} Part~(i) is straightforward. To prove (ii), note first that for every $h\in L_1$ 
we have $f_n(h)\leq\int_0^1|h| \, d\mu$ for all $n\in \Nat$, so the increasing sequence 
$(f_n(h))$ is bounded and converges to its supremum. Let us denote 
$$
	\phi(h):=\sup_{n\in \Nat} f_n(h)=\lim_{n\to \infty}f_n(h) \quad\mbox{for }h\in L_1.
$$
We want to show that $\phi(h)=\int_0^1|h| \, d\mu$ for every $h\in L_1$. 

Observe first that this equality is clear whenever $h$ is of the form
\begin{equation}\label{eqn:simple}
	\sum_{j=1}^{2^n} a_j \chi_{I_{n,j}}
\end{equation}
for some $n\in \Nat$ and some $a_1,\dots,a_{2^n}\in \erre$. To prove the equality for arbitrary $h\in L_1$ it 
is enough to show that $\phi: L_1\to \erre$ is $\|\cdot\|_1$-continuous (because simple functions as in~\eqref{eqn:simple} are dense in~$L_1$).
In fact, we will check that
\begin{equation}\label{eqn:continuity}
	|\phi(h)-\phi(h')|\leq \|h-h'\|_1
	\quad\mbox{for every }h,h'\in L_1.
\end{equation}
Indeed, given any $n\in \Nat$, we have
\begin{eqnarray*}
	|f_n(h)-f_n(h')| &=& \left| \sum_{j=1}^{2^n} \left( \Big|\int_{I_{n,j}} h \, d\mu\Big| -\Big|\int_{I_{n,j}} h' \, d\mu\Big| \right)	\right| \\
	&\leq& \sum_{j=1}^{2^n} \left| \Big|\int_{I_{n,j}} h \, d\mu\Big| -\Big|\int_{I_{n,j}} h' \, d\mu\Big| \right| \\
	&\leq& \sum_{j=1}^{2^n} \left| \int_{I_{n,j}} h \, d\mu -\int_{I_{n,j}} h' \, d\mu \right| \\
	&\leq& \sum_{j=1}^{2^n} \int_{I_{n,j}} |h-h'| \, d\mu=\|h-h'\|_1.
\end{eqnarray*}
As $n\in\Nat$ is arbitrary, \eqref{eqn:continuity} holds and the proof is finished.
\end{proof}

\begin{thm}\label{t:nofatou}
The norm of $FBL[L_1]$ fails the Fatou property.
\end{thm}
\begin{proof} We use the notation of Lemma~\ref{l:fn}. By considering the natural
inclusion of~$L_\infty=L_1^*$ in~$L_1$, each $f_n$ can be seen as an element of $H[L_1]_+$. In fact, we have
$f_n\in FBL[L_1]_+$ and $\|f_n\|_{FBL[L_1]}=1$, because
$$
	f_n=\sum_{j=1}^{2^n}|\delta_{\chi_{I_{n,j}}}|
$$
and 
$$
	1=f_n(\chi_{[0,1]}) \leq \|f_n\|_\infty \leq \|f_n\|_{FBL[L_1]}
	\leq 
	\sum_{j=1}^{2^n}\|\delta_{\chi_{I_{n,j}}}\|_{FBL[L_1]}=
	\sum_{j=1}^{2^n}\|\chi_{I_{n,j}}\|_{1} = 1.
$$

Fix any $g\in FBL[L_1]_+$ with $\|g\|_{FBL[L_1]}>1$, and let $\tilde{f}_n:=g\wedge f_n \in FBL[L_1]_+$ for all $n\in \Nat$.  
The sequence $(\tilde f_n)$ is increasing (by Lemma \ref{l:fn}), bounded above by~$g$ and 
$$
	\sup_{n\in \Nat}\|\tilde f_n\|_{FBL[L_1]}\leq \sup_{n\in \Nat}\| f_n\|_{FBL[L_1]} =1.
$$

We claim that $\sup_{n\in \Nat} \tilde f_n= g$ in $FBL[L_1]$. 
Indeed, fix $h\in L_\infty$ and let 
$$
	K:=\|g\|_{FBL[L_1]}\|h\|_\infty+\|h\|_1+1.
$$ 
Define 
$$
	h_j:=h+K r_j \in L_\infty \quad\mbox{for all }j\in \Nat
$$ 
(where $r_j$ denotes the $j$-th Rademacher function)
and observe that the sequence $(h_j)$ is $w^*$-convergent to~$h$ 
(as $(r_j)$ is $w^*$-null in~$L_\infty$). Since $g:L_\infty \to \erre_+$ is $w^*$-continuous on bounded sets (Lemma~\ref{lem:continuity2}), 
we have $g(h_j)\to g(h)$ as $j\to \infty$, so in particular there is $j_0\in \Nat$ 
such that for every $j\geq j_0$ we have
\begin{multline}\label{eqn:K}
	g(h_j)\leq g(h)+1 \leq \|g\|_{FBL[L_1]}\|h\|_\infty+1=  K-\|h\|_1\leq\int_0^1|h_j|d\mu.
\end{multline}
(the second inequality being a consequence of Remark~\ref{rem:PointwiseBounded}).

Now take any $\varphi\in FBL[L_1]$ satisfying $\varphi\geq \tilde f_n$ for every $n\in\mathbb N$. By Lemma \ref{l:fn}, for every $j\geq j_0$ we have
$$
	\varphi(h_j)\geq \sup_{n\in \Nat} \tilde f_n(h_j)=g(h_j)\wedge\int_0^1|h_j|d\mu\stackrel{\eqref{eqn:K}}{=}g(h_j).
$$
Since $\varphi$ and $g$ are $w^*$-continuous on bounded sets (Lemma~\ref{lem:continuity2}), it follows that $\varphi(h)\geq g(h)$.

Therefore, we have $\sup_{n\in \Nat} \tilde f_n= g$ in $FBL[L_1]$. 
Since $\sup_{n\in \Nat}\|\tilde f_n\|_{FBL[L_1]}\leq 1$ and $\|g\|_{FBL[L_1]}>1$, the Fatou property cannot hold.
\end{proof}

\section{An application to weakly compactly generated Banach lattices}\label{s:wcg}

The purpose of this section is to give a negative answer to Diestel's question mentioned  
in the introduction (Problem~\ref{problem:Diestel}). Interestingly enough, that question can be equivalently rephrased by asking:

\begin{prob}\label{problem:AntonioPedro}
If $E$ is a WCG Banach space, is $FBL[E]$ WCG as well?
\end{prob}

The equivalence of Problems~\ref{problem:Diestel} and~\ref{problem:AntonioPedro} follows at once from the following remarks:

\begin{rem}\label{r:E-WCG-implies-FBL-LWCG}
If $E$ is a WCG Banach space, then $FBL[E]$ is LWCG. 
\end{rem}
\begin{proof}
Let $K \sub E$ be a weakly compact set such that $E=\overline{{\rm span}}(K)$. Then  
the sublattice generated by the weakly compact set $\phi_E(K)=\{\delta_x:x\in K\}$ is dense in~$FBL[E]$.
\end{proof}

\begin{rem}\label{r:}
Let $X$ be an LWCG Banach lattice. Then there exist a WCG Banach space $E$ and 
an operator $\hat T:FBL[E] \to X$ with dense range.
\end{rem}
\begin{proof}
Let $K \sub X$ be a weakly compact set such that the sublattice generated by~$K$ is dense in~$X$.
Define $E:=\overline{{\rm span}}(K) \sub X$. Then there is a lattice homomorphism $\hat T:FBL[E]\to X$ 
such that $\hat T\circ \phi_E$ is the identity on~$E$ (Theorem~\ref{t:fblb}). The sublattice generated by $E$ in~$X$ 
is contained in the range of~$\hat T$, hence the range of $\hat T$ is dense in $X$.
\end{proof}

So we might ask whether $FBL[c_0(\Gamma)]$ or $FBL[\ell_p(\Gamma)]$ ($1<p<\infty$) are WCG for uncountable~$\Gamma$. 
We will next see that $FBL[\ell_p(\Gamma)]$ is not WCG whenever $\Gamma$ is uncountable and $1<p\leq 2$, 
hence answering in the negative Diestel's question.

\begin{thm}\label{l2wcg}
Let $\Gamma$ be a non-empty set and $1< p\leq 2$. Then $FBL[\ell_p(\Gamma)]$ contains a subspace isomorphic to~$\ell_1(\Gamma)$. 
\end{thm}

\begin{proof}
Let $(e_\gamma)_{\gamma\in\Gamma}$ denote the unit vector basis of $\ell_p(\Gamma)$. 
We will prove that the family $\{|\delta_{e_\gamma}|:\gamma\in \Gamma\} \sub S_{FBL[\ell_p(\Gamma)]}$ 
is equivalent to the unit vector basis of $\ell_1(\Gamma)$. 

Let $i: \ell_p \to \ell_2$ be the formal inclusion operator (so that $\|i\|=1$)
and let $j: \ell_2 \to L_1$ be the isomorphic embedding satisfying $\|j\|=1$ and $j(e_n)=r_n$
for all $n\in \Nat$, where $(e_n)$ is the unit vector basis of~$\ell_2$ and $r_n$ denotes the $n$-th Rademacher function
(see e.g. \cite[Theorem~6.2.3]{alb-kal}). 

Given any finite set $A=\{\gamma_1,\dots,\gamma_n\} \sub \Gamma$, let $P_A:\ell_p(\Gamma)\to \ell_p$ be the operator defined by
$$
	P_A((x_\gamma)_{\gamma\in \Gamma}):=(x_{\gamma_1},x_{\gamma_2},\dots,x_{\gamma_n},0,0,\dots)
$$ 
(for $A=\emptyset$ we define $P_\emptyset:=0$) and consider
$$
	T_A:=j\circ i \circ P_A: \ell_p(\Gamma)\rightarrow L_1,
$$
so that $\|T_A\| \leq 1$. Let $\hat{T}_A:FBL[\ell_p(\Gamma)]\to L_1$ be the unique lattice homomorphism extending $T_A$, which
satisfies $\|\hat{T}_A\|=\|T_A\|$ (Theorem~\ref{t:fblb}).
Define $\xi_A^*:=(\hat{T}_A)^*(\chi_{[0,1]})\in FBL[\ell_p(\Gamma)]^*$
and note that $\|\xi_A^*\|\leq 1$. Moreover, for each $\gamma\in \Gamma$ we have
\begin{equation}\label{eqn:TA}
	\big\langle \xi_A^*,|\delta_{e_\gamma}|\big\rangle=
	\big\langle \chi_{[0,1]},\hat{T}_A(|\delta_{e_\gamma}|)\big\rangle=
	\big\langle \chi_{[0,1]},|T_A(e_\gamma)|\big\rangle
	=\left\{
	\begin{array}{ccc}
 	1 &   & \textrm{ if }\gamma\in A  \\
  	&   &   \\
 	0 &   & \textrm{ if }\gamma\not\in A.
	\end{array}
	\right.
\end{equation}

Finally, take any finite non-empty set $B \sub \Gamma$ and pick $a_\gamma\in \erre$ for each $\gamma\in B$. 
Write $B_+:=\{\gamma\in B: a_\gamma>0\}$ and $B_-:=\{\gamma\in B: a_\gamma<0\}$. 
From~\eqref{eqn:TA} it follows that
$$
	2\Big\|\sum_{\gamma\in B} a_\gamma|\delta_{e_\gamma}|\Big\|_{FBL[\ell_p(\Gamma)]}
	\geq 
	\xi_{B_+}^*\Big(\sum_{\gamma \in B} a_\gamma|\delta_{e_\gamma}|\Big) - 
	\xi_{B_-}^*\Big(\sum_{\gamma \in B} a_\gamma|\delta_{e_\gamma}|\Big)=\sum_{\gamma\in B} |a_\gamma|.
$$
This shows that $\{|\delta_{e_\gamma}|:\gamma\in \Gamma\}$ 
is equivalent to the unit vector basis of $\ell_1(\Gamma)$, hence
$\overline{{\rm span}}(\{|\delta_{e_\gamma}|:\gamma\in \Gamma\})$
is isomorphic to~$\ell_1(\Gamma)$.
\end{proof}

\begin{cor}\label{c:Diestel}
Let $\Gamma$ be an uncountable set and $1 < p \leq 2$. Then $FBL[\ell_p(\Gamma)]$ is LWCG but it is
not isomorphic to a subspace of a WCG Banach space.
\end{cor}
\begin{proof}
Bear in mind that $\ell_1(\Gamma)$ does not embed isomorphically into any WCG Banach space. 
This can be deduced, for instance, from the fact that subspaces of WCG Banach spaces
are weakly Lindel\"of (see e.g. \cite[Theorem 14.31]{fab-ultimo}), 
while $\ell_1(\Gamma)$ is not weakly Lindel\"of (see e.g. \cite[Proposition~5.11]{edgar2}).
\end{proof}

We do not know whether the spaces $FBL[c_0(\Gamma)]$ or $FBL[\ell_p(\Gamma)]$ for $2<p<\infty$ and uncountable~$\Gamma$
are WCG. In fact, we do not know any non-separable Banach space $E$ for which $FBL[E]$ is WCG.

\section*{Acknowledgements}
The research of A.~Avil\'{e}s and J.~Rodr\'{i}guez was supported by 
Ministerio de Econom\'{i}a y Competitividad and FEDER (project MTM2014-54182-P)
and Fundaci\'{o}n S\'{e}neca (project 19275/PI/14).
The research of P.~Tradacete was supported by Ministerio de Econom\'{i}a y Competitividad
(projects MTM2016-75196-P and MTM2016-76808-P) and Grupo UCM 910346.

\def\cprime{$'$}

\providecommand{\MR}{\relax\ifhmode\unskip\space\fi MR }

\end{document}